\documentclass[12pt]{amsart}
\usepackage[margin=1in]{geometry}
\usepackage{amscd,amsmath,bm,amsthm, amssymb}
\usepackage{enumerate}
\usepackage[linktocpage=true, pdfencoding=auto, psdextra]{hyperref}

\usepackage{mathtools}
\usepackage{xcolor}
\usepackage{tikz}
\usetikzlibrary{matrix}
\usepackage{amsfonts}
\usepackage{enumitem}

\usepackage{mathtools}
\newtheorem{thm}{Theorem}[section]

\newtheorem{lem}[thm]{Lemma}
\newtheorem{prop}[thm]{Proposition}

\theoremstyle{definition}
\newtheorem{defin}[thm]{Definition}

\theoremstyle{remark}

\newtheorem*{theorem*}{Theorem}

\def\VR{VR(Q_n;r)}
\def\nz{\mathbb{Z}}

\begin{document}
	\title[Connectivty of Vietoris-Rips complexes]{On the Connectivity of the Vietoris-Rips Complex of a Hypercube Graph}

	\author[M.~Bendersky]{Martin~Bendersky}
	\address{Department of Mathematics, Hunter College, CUNY,   695 Park Avenue New York, NY 10065, U.S.A.}
	\email{mbenders@hunter.cuny.edu}
	
	\author[J.~Grbi\' c]{Jelena~Grbi\' c}
	\address{School of Mathematical Sciences, University of Southampton, SO17 1BJ Southampton, UK}
	\email{J.Grbic@soton.ac.uk}

	
	
	\begin{abstract}
		We bring in the techniques of independence complexes and the notion of total dominating sets of a graph to bear on the question of the connectivity of the Vietoris-Rips complexes $VR(Q_n; r)$ of an $n$-hypercube graph. We obtain a lower bound for the connectivity of $VR(Q_n; r)$ for an arbitrary $n$-dimension hypercube and at all scale parameters $r$. The obtained bounds disprove the conjecture of Shukla that $\VR$ is $r$-connected.
		
	\end{abstract}
	\maketitle

	\section{Introduction}
	Let $Q_n$ be the set of vertices of the $n$-hypercube $\mathbb I^n$, endowed with the shortest path metric.  Thinking of the natural realisation of the hypercube $\mathbb I^n$ as the Cartesian product space $[0,1]^n$, to each vertex in $Q_n$, we associate a binary string of length $n$. This set of $2^n$ binary strings, which we will also denote by $Q_n$, can be equipped with the Hamming distance, that is, with the metrics which measures the number of coordinates at which two binary strings of length $n$ differ.
	
	Given the set $Q_n$ with the Hamming distance, $d_H$ and a rational number $r\geq 0$, the {\it Vietoris-Rips simplicial complex} $\VR$ of the $n$-hypercube graph at a scale parameter $r$ has for its vertex set $Q_n$ and $\sigma\subset Q_n$ as a simplex if the set diameter of $\sigma$ is at most $r$.
	There has been recent increased interest in and progress towards understanding the topology of the Vietoris-Rips complex $VR(Q_n; r)$ of the $n$-hypercube graph with the shortest path metric at scale parameter $r$ (\cite{aa}, \cite{av}, \cite{f}, \cite{s}). 
	Adamaskez and Adams \cite{aa} described the homotopy type of $VR(Q_n;2)$, Shukla  \cite{s} showed that the cohomology of $VR(Q_n;3)$ is concentrated in dimensions $4$ and $7$.   The homotopy type of  $VR(Q_n;3)$ is fully described in \cite{f}.  In~\cite{s} Shukla conjectured that for $n\geq r+2$, the reduced homology $\widetilde{H}_i(VR(Q_n; r);\nz)\neq 0$ if and only if $i=r+1$ or $r=2^r-1$ implying that $VR(Q_n;r)$ is $r$-connected. Using computer calculations, in \cite {av} and \cite{b}, it was shown that $VR(Q_6;4)$ is $6$ connected.  
	
	In this note we provide a method, using the total domination invariant of a graph, to determine a lower bound for the connectivity of the Vietoris-Rips complexes $\VR$ for an arbitrary $n$ and $r$. This approach provides an infinite family of counterexamples to Shukla's conjecture.

	To this end we recall some basic notions and results related to total dominating sets of a graph and indipendence complexes og a graph.
	
	A {\it graph} $G = (V(G);E(G))$ consists of a nonempty set $V(G)$ of objects called vertices together with a (possibly empty) set $E(G)$ of unordered pairs of distinct vertices of $G$ called edges. We will only consider simple graphs, those without directed edges or loops. An {\it isolated vertex} is a vertex that is not an endpoint of any edge.

	\begin{defin}[\cite{hy}] 
		A  {\it total dominating set} of a graph $G$ with no isolated vertices is a subset $S$ of vertices of $G$ such that every vertex is connected by an edge to a vertex in $S$.
	\end{defin} 
	
	It is worth emphasising that each vertex in $S$ also needs to be connected by an edge to another vertex in $S$. For example, in the figure below the set $S=\{1,2\}$ is not a total domination set since the point $\{1\}$ is not connected to any  point in $S$, while the set $\{0,2\}$ is a domination set since every point in the vertex set of $G$ is connected to some point in $S$.
	
	\setlength{\unitlength}{1.5cm}
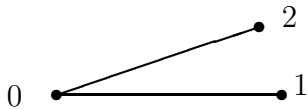
	\begin{picture}(10,4)(5,1)
		\thicklines
		\put(8,3){\circle*{0.1}}
		\put(7.56,2.9){$0$}
		\put(10.1,3){$1$}
		\put(10,3){\circle*{0.1}}
		\put(8,3){\line(3,1){1.8}}
		\put(8,3){\line(3,0){2}}
		\put(9.8,3.6){\circle*{0.1}}
		\put(10,3.6){$2$}
	\end{picture}
	
	\vspace{-1.5cm}
	\begin{defin}
		The {\it total domination number} of G, denoted by $\gamma_t(G)$, is the minimal cardinality amongst all total domination sets of $G$.
	\end{defin} 
	
	Our interest in $\gamma_t(G)$ is inspired by a theorem of Chudnovsky \cite{c} which relates the total domination number of a graph to the connectivity of its independence complex. 
	
	Let $G=(V,E)$ be a graph. 
	The {\it complement} of $G$, denoted by $G^c=(V, E^c)$, is the graph with the vertex set $V$ and $e$ is an edge of $G^c$, that is, $e\in E^c$, if $e$ is not an edge in $G$.
	The {\it  independence complex of $G$}, denoted by $I(G)$, is the clique complex of $G^c$.

	\begin{thm}[\cite{c}, \cite{m}] \label{thm:ch} 
		If $\gamma_t(G) > 2k$, then $I(G)$ is $(k-1)$-connected.
	\end{thm}

	For integers $n$ and $r$, and the set $Q_n$ of binary strings of length $n$ equipped with the Hamming distance $d_H$, we consider the graph $G_{n,r}$ whose vertex set is $Q_n$ and there is an edge in $G_{n,r}$ between two vertices, $a,b \in Q_n$ if and only if $d_H(a,b) \leq r$. Notice that the Vietoris-Rips complex $\VR$ is the click complex of $G_{n,r}$. 
	
	\section{Connectivity of Vietoris-Rips Complex}
	
	In this section we provide a lower bound for the connectivity of $\VR$.  The basis of the bound is a crude estimate for the total domination number. Recall that the {\it order of a graph} is the number of vertices of the graph, while the {\it order of a vertex} is the number of graph edges meeting at that vertex.  The maximal order of any vertex  in  a graph $G$ is denoted by $\Delta(G)$.
	
	\begin{thm}[\cite{hy}Theorem 2.11]\label{thm:hy}
		If $G$ is a graph of order $m$ with no isolated vertices, then 
		\[
		\gamma_t(G) \geq \frac{m}{\Delta(G)}.
		\]
	\end{thm}
	The estimate in Theorem \ref{thm:hy} is an immediate consequence of the definition of a total domination set.   Namely,  every vertex in $G$  belongs to the open neighborhood of at least one vertex in a minimal total domination set,  that is,  $\gamma_t(G) \Delta(G) \geq m$.  Equality occurs when each vertex is connected to a point in the total dominating set by  one edge.  The figure below is  an example of such a graph..  This graph has total dominating set $\{S_1, \cdots S_{\gamma_t(G)} \}$.  The order of each point in the domination set is $\Delta(G)$.  The graph   is the union of isomorphic components each of which has two points from the dominating set and $2\Delta(G)$ vertices.  The estimate in Theorem \ref{thm:hy} becomes a strict inequality if some of the vertices at the termini of the vertical edges coincide. 
	
	
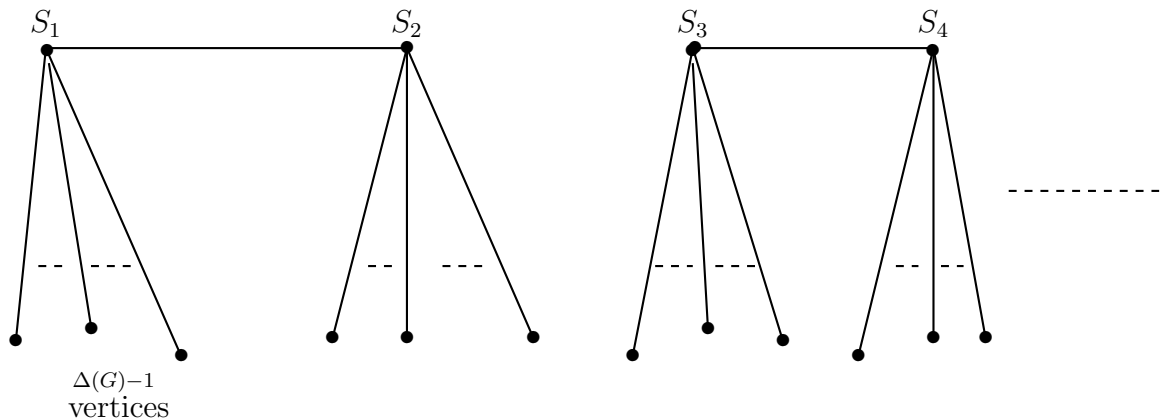
\begin{center}
		\begin{tikzpicture}
			
			\draw [ thick ] (-1.8,1.9)node[above]{$S_1$} --( 0,-2.2)node[below left] { $\overset{\Delta(G)-1}  {\mbox{ vertices}}$}
			node[pos=1.00]{$\bullet$}
			node[pos=0.01]{$\bullet$};

			\draw [ thick] (-1.2,-1.8) -- (-1.76,1.7)
			node[pos=-0.01]{$\bullet$};
			\draw [, thick ] (-1.8,1.9) --( -2.2,-2)
			node[pos=1.00]{$\bullet$};
			
			\draw[ thick ] (-1.8,1.9) --( 3,1.9)node[above]{$S_2$}
			node[pos=1.00]{$\bullet$};

			\draw [, thick ] (2,-2) --( 3,1.9)
			node[pos=0.01]{$\bullet$};
			\draw [ thick] (3,-2) -- (3,1.9)
			node[pos=0.01]{$\bullet$};
			\draw [, thick ] (4.7,-2) --( 3,1.9)
			node[pos=0.01]{$\bullet$};

			\draw [ thick ] (6.8,1.9)node[above]{$S_3$} --(6,-2.2)
			%	node[below left] { $\overset{\Delta(G)-1}  {\mbox{ vertices}}$}
			node[pos=1.00]{$\bullet$}
			node[pos=0.01]{$\bullet$};

			\draw [ thick] (7,-1.8) -- (6.8,1.7)
			node[pos=-0.01]{$\bullet$};
			\draw [, thick ] (6.8,1.9) --( 8,-2)
			node[pos=1.00]{$\bullet$};

			\draw [ thick ] (10,1.9)node[above]{$S_4$} --(9,-2.2)
			node[pos=1.00]{$\bullet$}
			node[pos=0.01]{$\bullet$};
			
			\draw [, thick ] (6.8,1.9) --( 10,1.9)
			node[pos=0.01]{$\bullet$};
			\draw [ thick] (10,-2) -- (10,1.9)
			node[pos=0.01]{$\bullet$};
			\draw [, thick ] (10.7,-2) --( 10,1.9)
			node[pos=0.01]{$\bullet$};
			
			\draw[ thick,dashed] (11,0) -- (13,0);
			
			\draw[ thick, dashed] (-1.9,-1) -- (-1.5,-1);
			\draw[ thick, dashed] (-1.2,-1) -- (-.6,-1);
			\draw[ thick, dashed] (2.8,-1) -- (2.4,-1);
			\draw[ thick, dashed] (4,-1) -- (3.4,-1);
			\draw[ thick, dashed] (6.3,-1) -- (6.8,-1);
			\draw[ thick, dashed] (7.1,-1) -- (7.7,-1);
			\draw[ thick, dashed] (9.5,-1) -- (9.8,-1);
			\draw[ thick, dashed] (10.1,-1) -- (10.4,-1);

		\end{tikzpicture}
	\end{center}
	
	As $\VR$ is the independence complex of $G_{n,r}^c$, we will apply Theorems \ref{thm:ch} and  \ref{thm:hy} to the graph $G_{n,r}^c$. 
	
	\begin{lem}\label{lem:order}  The order of any vertex of $G_{n,r}^c$ is
		\[
		\sum_{i=r+1}^n \binom{n}{i}. 
		\]
	\end{lem}
	
	\begin{proof}  
		For any vertex $v \in Q_n$, there is are $\binom{n}{i}$ edges from $v$ to vertices $w$ such that $d_H(v,w)=i$. Thus there are $\sum_{i=r+1}^n \binom{n}{i}$ edges from any vertex $v$ to vertices with  Hamming distance $r+1$ or greater.
	\end{proof}

	\begin{thm} \label{thm:main} 
		Let $\alpha_{n,r} = \dfrac{2^{n-1}}{\sum_{i=r+1}^n \binom{n}{i} }$.  Then $\VR$ is $(k-1)$-connected, where
		\[
		k=	\begin{cases}
			\Big[ \alpha_{n,r} \Big] & \mbox{ if } \alpha_{n,r} \mbox{ is not an integer } \\
			\alpha_{n,r}  -1 & \mbox{ otherwise.} 
		\end{cases}
		\]
	\end{thm}
	\begin{proof}
		By Theorem \ref{thm:hy} and Lemma \ref{lem:order}, 
		\[
		\gamma_t(G_{n,r}^c) \geq 2  \alpha_{n,r} > 2k.
		\] 
		By Theorem \ref{thm:ch}, $\VR=I(G^c_{n,r})$ is $(k-1)$-connected.
	\end{proof}
	
	Theorem \ref{thm:main} is a coarse bound.  For example, for a fixed $r$, the lower bound on the connectivity approaches $0$ as $n \to \infty$.  In spite this, this estimate gives an infinite family of counterexamples for which the connectivity of $\VR$ given by Theorem~\ref{thm:main} is greater than $r+1$.  Moreover, this estimate suggests that the connectivity of $\VR$ does not grow linearly with respect to either $n$ or $r$ and illustrates that for some $n$ and $r$, $\VR$ is highly connected. 
	
	We least a few examples.
	\begin{center}
		\begin{tabular}{c c c c||} 
			
			$n$ & $r$ & connectivity\\ [0.5ex] 
			\hline
			7 & 5 & 6 \\ 
			\hline
			8 & 6 & 13 \\
			\hline
			9 & 7 & 24  \\
			\hline
			12 & 10 & 156 \\
			\hline
			18 & 15 & 761\\
			\hline
			18 & 16 & 6897\\
			\hline
			20 & 16 & 387\\
			\hline
			20&  18 & 24964\\
		\end{tabular}
	\end{center}
	
	We remark that the connectivity of Theorem \ref{thm:main} does not violate the non triviality of $H_{2^r-1}(\VR;\nz)$, \cite{av}.

	The large connectivity we obtain is based on a simple estimate of the total domination.   We expect that the  total domination of the graphs $G^c_{n,r}$ are significantly larger than provided by Theorem \ref{thm:hy}. For example, Feng's computer calculation, \cite{av}, shows that $VR(Q_6,4)$ is $6$-connected and $H_7(\VR;\nz) \neq 0$.  This implies that $\gamma_t(G^c_{6,4})\leq 16.$ By Theorem~\ref{thm:main}, we can only conclude that the total domination of $G^c_{6,4}$ is greater than equal to 8 and thus $VR(Q_6,4)$ is at least $3$-connected. It would be interesting to know how close the total domination number of $G^c_{6,4}$ is to $16$.
	
	In general, describing total domination sets for an arbitrary graph is a difficult problem. Hypercube graphs are rich in symmetry and one hopes that the geometry of hypercubes and their duals, cross-polytopes could shed light on a particularly nice structures of their total domination sets.
	We do have the following connection between certain non-trivial homology classes in $H_*(\VR;\nz)$ and total domination sets.
	
	\begin{prop} 
		Suppose a non-trivial homology class $a \in H_{m-1}(\VR;\nz) $ is represented by $\alpha$,  the boundary of a cross-polytope on $2m$ vertices.  Then the set of vertices of $\alpha$ is a total domination set of $G_{n,r}^c$
	\end{prop}
	\begin{proof}
		Suppose the vertices of $\alpha$ are $\mathcal{C}=\{ v_1,w_1,v_2,w_2, \ldots, v_m,w_m\}$, where  $\{v_i,w_i\}$ is not an edge in $\VR$ and there is an edge connecting all other vertices in $\mathcal{C}$.  The vertices of $\alpha$ represent a matching edge set in $G_{n,r}^c$ (that is, a set of edges without common vertices).  If $\mathcal{C}$ was not a total domination set, then there is a vertex $u$ which is not connected to any vertex of $\mathcal{C}$ in $G_{n,r}^c$.   This means that  $u *\alpha$, the cone on $\alpha$, is a subcomplex of $\VR$ which contradicts the non-triviality of $\alpha$.
	\end{proof}

\bibliographystyle{amsalpha}

\end{document}